\documentclass[12pt]{amsart}
\usepackage{epsfig,amsmath,amsfonts,latexsym}
\usepackage[usenames,dvipsnames]{color}
\usepackage{palatino}
\usepackage[hypcap=false]{caption}
\usepackage{color}
\usepackage{marginnote}

\newcommand{\red}[1]{\textcolor{black}{#1}}

\usepackage[normalem]{ulem}
\usepackage[top=2cm, bottom=1.3cm, left=5cm, right=0.5cm, heightrounded,
  marginparwidth=2.8cm, marginparsep=3mm]{geometry}
\usepackage[hyperfootnotes=false]		{hyperref}
\hypersetup{
  colorlinks,
  citecolor=Purple,
  linkcolor=Black,
  urlcolor=Green}
  \usepackage[normalem]{ulem}

\theoremstyle{plain}
\newtheorem{dummy}{anything}[section]
\newtheorem{theorem}[dummy]{Theorem}

\newtheorem{lemma}[dummy]{Lemma}

\newtheorem{proposition}[dummy]{Proposition}

\newtheorem{corollary}[dummy]{Corollary}

\theoremstyle{definition}
\newtheorem{definition}[dummy]{Definition}

\newtheorem{remark}[dummy]{Remark}

\theoremstyle{remark}

\newcommand{\del}{\partial}

\newcommand{\Z}{\mathbb{Z}}
\newcommand{\R}{\mathbb{R}}
\newcommand{\C}{\mathbb{C}}

\newcommand{\SLZ}{\textrm{SL}_2(\Z)}

\newcommand{\coupe}{\setminus\!\!\setminus}
\newcommand{\id}{\mathrm{id}}

\def\a{\alpha}

\def\t{\mu}
\def\l{\lambda}
\newcommand{\Sp}{\mathbb{S}}

\newcommand{\ob}{\mathop{\rm OB}\nolimits}
\newcommand{\bn}{\mathop{\rm bn}\nolimits}

\renewcommand{\t}{\mathbf t}

\def\o{\omega}
\def\S{\Sigma}
\def\t{\theta}
\begin{document}

\title[]{Complex vs convex Morse functions \linebreak and geodesic open books}

\author{Pierre Dehornoy}
\address{Univ. Grenoble Alpes, CNRS, Institut Fourier, F-38000
Grenoble, France}
\email{pierre.dehornoy@univ-grenoble-alpes.fr}

\author{Burak Ozbagci}
\address{Department of Mathematics \\ Ko\c{c} University \\ Istanbul, Turkey}
\email{bozbagci@ku.edu.tr}


\begin{abstract} Suppose that $\S$ is a closed and oriented surface equipped with a Riemannian metric. In the literature, there are three seemingly distinct constructions of open books on the unit (co)tangent bundle of $\S$,  having  complex, contact, and dynamical flavors, respectively.
Each one of these constructions is based on either an admissible divide or an ordered Morse function on $\S$. 
We show  that the resulting open books are pairwise isotopic provided that the ordered Morse function is adapted to the admissible divide on~$\S$. 
Moreover, we observe that if $\S$ has positive genus,  then none of these open books are planar and furthermore,  we determine the only cases when they have genus one pages.
\end{abstract}

\maketitle

\section{Introduction}\label{int}

Let $\S$ be a closed and oriented surface. 
The bundle of cooriented lines tangent to  $\S$ (aka the space of cooriented contact elements), which we denote by $V(\S)$ in this paper,  is defined as the set of pairs $(q, L)$ where $q \in \S$ and $L$ is a cooriented line in $T_q\S$.
\red{When $\S$ is equipped with a Riemannian metric,}
one can identify $V(\S)$ with the unit tangent bundle $ST\S$ as well as with the unit cotangent bundle $ST^*\S$.   As an oriented $3$-manifold,  $ST^*\S$ (resp. $ST\S$)  is diffeomorphic to the circle bundle over~$\S$ with Euler number $-\chi (\S)$ (resp. $\chi (\S)$). We orient the circle bundle $V(\S)$ so that it is fiber and orientation preserving diffeomorphic to $ST^*\S$.

Let $DT^*\S$ denote the disk cotangent bundle, which is a disk bundle over $\S$ with Euler number $-\chi (\S)$, whose boundary is  $ST^*\S$. Let $\lambda$ denote the Liouville  $1$-form on $T^*\S$. Under the identification of $ST^*\S$ with $V(\S)$,  the canonical contact structure $\xi =\ker \lambda|_{ST^*\S}$  on $ST^*\S$ coincides with the canonical contact structure on $V(\S)$, which we again denote by $\xi$ throughout the article.

A \emph{divide} $P \subset \S$ is a generic immersion of the disjoint union of finitely many copies of the unit circle. Here generic means that the image has neither self tangencies nor triple intersections. An edge of $P$ is the closure of a component of $P \setminus \{\mbox{double points}\}$ and a region of $P$ is a component of $\S  \setminus P$. 

A \emph{divide} $P$ is said to be  \emph{admissible} if it is connected,  each region of $P$  is simply connected and $\S \setminus P$ admits a black-and-white coloring so that the two regions on opposite sides of an edge in $P$ have opposite colors. See Figure~\ref{F:Divide} below, for an example of an admissible divide on a genus two surface.

A {\em Morse function} $f: \S \to \R$ is a smooth function that has finitely many  non-degenerate critical points. The Morse function $f$ is said to be {\em ordered} if the higher the index of a critical point the greater its value.
 
A {\em Morse function} $f: \S \to \R$ is {\em adapted} to a given admissible divide $P \subset \S$ if one has~$P= f^{-1}(0)$,  each double point of $P$ corresponds to a critical point of~$f$ of index~$1$, and  each black (resp. white) region of $\S \setminus P$ contains exactly one index~$2$ (resp.~$0$) critical point of~$f$. 

We recall in Section~\ref{sec: ACampo} that for every admissible divide~$P \subset \S$, there exists an ordered Morse function $f: \S \to \R$ adapted to $P$,  whose isotopy class is uniquely determined by the pair $(\S, P)$~\cite{aca, i}.

An admissible divide $P \subset \S$ is said to be {\em convex} with respect to the given metric on $\S$,  if $P$  is a set of geodesics on $\S$, every geodesic on~$\S$ meets $P$ in bounded time, and every region of~$\S\setminus P$ can be foliated by concentric closed curves with non-vanishing curvature.

An open book for a closed and oriented $3$-manifold $M$ is a pair $(B, \pi)$ so that  $B$ is an oriented link in $M$ called the {\em binding} of the open book and $\pi : M \setminus B \to  S^1$ is a fibration such that $\pi^{-1}(\theta)$ is the interior of a compact surface $S_\theta \subset M$ with $\partial  S_\theta= B$ for all $\theta \in S^1$, and in a tubular neighborhood of the form~$D^2\times S^1$ of~$B$ the fibration is given by~$(r,\theta, t)\mapsto\theta$. The surface $S=S_\theta$, for any $\theta$, is called the {\em page} of the open book.

Let $(B, \pi)$ and $(B', \pi')$ be two open books for closed and oriented $3$-manifolds $M$ and~$M'$, respectively. An {\em isomorphism} from $(B, \pi)$ to $(B', \pi')$  is a diffeomorphism $\varphi : M \to M'$  such that $B' =\varphi (B)$ and $\pi = \pi' \circ \varphi$. When $M = M'$, an {\em isotopy}  from~$(B, \pi)$ to~$(B', \pi')$  is an isomorphism  $\varphi$ which is isotopic to the identity through diffeomorphisms of $M$.

\begin{theorem} \label{thm: main} 
Suppose that $P$ is an admissible  divide on a closed and oriented surface~$\S$ equipped with a Riemannian metric, and let $f: \S \to \R$ be  an ordered Morse function adapted to $P$. 
Then the open books described in the constructions (1) - (3) below belong to the same isotopy class under the pairwise identifications of the oriented 3-manifolds $ST\S$, $ST^*\S$ and~$V(\S)$, which we describe in Section~\ref{sec: identify}.

 {\em  (1)  (A'Campo \cite{ac}  $\&$ Ishikawa \cite{i}) } The Morse function $f: \S \to \R$ can be extended to a complex Morse function  $f_{\C}: T\S \to \C$, which restricts to an {\em achiral}\footnote{There is orientation error in \cite[Lemma 2.6]{i}: The Lefschetz fibration $DT\S \to D^2$  Ishikawa constructs  must be {\em achiral}  or  equivalently, it must be considered on the cotangent (rather than tangent) disk bundle $DT^*\S$.}  Lefschetz fibration $DT\S \to D^2$.  There is an induced open book for the boundary  $ST\S$. \red{By reversing the orientation, there is a Lefschetz fibration $DT^*\S \to D^2$, which induces an open book for  $ST^*\S$.}

{\em (2) (Giroux \cite{g})} The Morse function $f: \S \to \R$ can be modified to an ordered\footnote{The fact that $f_{\xi}$ is ordered, provided that $f$ is ordered, was observed by Massot \cite{m} along with some other details of Giroux's construction \cite[Example 4.9]{g} restricted to  dimension three.} $\xi$-convex (or contact) Morse function  $f_{\xi}: V(\S) \to \R$, which induces an open book for $V(\S)$ that supports $\xi$.

{\em (3)  (Birkhoff  \cite{b} \& Fried \cite{f})}  If $P$ is convex, then  the lift of $P$  to $ST\S$ is the binding of a ''geodesic'' open book, that is an open book whose pages are negative\footnote{We use the term \emph{negative} since the natural orientation given by the geodesic flow of the binding of this open book is opposite to that induced from a natural orientation of the page.} Birkhoff cross sections of the geodesic flow.  \end{theorem}

In the above context of $\S$ a closed and oriented surface with a Riemannian metric, and an admissible divide $P \subset \S$, the construction described in item (1) determines a unique isotopy class of open books for $ST\S$ (and also  for $ST^*\S$). We call this isotopy class as the {\em A'Campo-Ishikawa open book} for $ST\S$ (and also for $ST^*\S$). For each ordered Morse function $f: \S \to \R$ adapted to~$P$, which is unique up to isotopy, the construction described in item (2) determines a unique isotopy class of open books which is  called the {\em Giroux open book} for  $V(\S)$. 
Finally, if the admissible divide $P \subset \S$ is convex with respect to the given metric on $\S$, then the construction described in item (3) determines a unique isotopy class of open books, which is called the {\em geodesic open book} for  $ST\S$. 
In the same vein, there is a {\em cogeodesic open book} for $ST^*\S$, whose pages are {\em positive} Birkhoff cross sections of the cogeodesic flow. 
Note that for each of the constructions in items (1)-(3), we make some auxiliary choices. 
However, Proposition~\ref{prop: onepageisotopic}, which does not seem to have appeared in this form in the literature before, guarantees that these auxiliary choices do not affect the isotopy class of the described open books, so that each of the constructions in items (1)-(3) describes a unique isotopy class of open books as we claimed above.

\begin{proposition}\label{prop: onepageisotopic}
\red{Two open books $(B, \pi)$ and $(B', \pi')$ for a closed and oriented $3$-manifold~$M$ are isotopic, provided that the 
 pages $\overline{\pi^{-1}(0)}$ and $\overline{\pi'^{-1}(0)}$ are isotopic in~$M$. }
\end{proposition}

\begin{remark} Two open books for a closed and oriented $3$-manifold $M$ with a common binding are not necessarily isomorphic~\cite[Example 2.2]{et06}.
\end{remark} 

The content of Theorem~\ref{thm: main} is that the A'Campo-Ishikawa open book, the Giroux open book and the (co)geodesic open book are mapped to each other under prescribed pairwise diffeomorphisms between $ST\S$, $ST^*\S$ and~$V(\S)$. The proof of Theorem~\ref{thm: main} also relies on Proposition~\ref{prop: onepageisotopic}.  Thanks to  Proposition~\ref{prop: onepageisotopic}, it suffices to compare one page of each of the considered open books in order to prove Theorem~\ref{thm: main}. It turns out that each of these open books have {\em four} particular pages, and we will show in Section~\ref{sec: proof} that there is an isotopy between these particular pages in all three constructions.

Let $(B, \pi)$ an open book for a closed an oriented $3$-manifold $M$.  A contact structure $\eta$ on $M$ is supported by $(B, \pi)$ if $\eta$  is  the kernel of a contact $1$-form $\a$ satisfying the following conditions: (i) $d\a$ is a positive area form on each page $S_\theta$ of the open book $(B, \pi)$ and (ii) $\a > 0$ on the binding $B$. In this case, we also say that the open book $(B, \pi)$ is adapted to the contact $3$-manifold $(M, \eta)$. 

Note that the \red{Giroux} open book supports the canonical contact structure $\xi$ on $V(\S) \cong ST^*\S$ by construction. Moreover,  the ``dual" of the geodesic open book \red{obtained by} the cogeodesic flow on $ST^*\S$, also supports $\xi$, since the cogeodesic flow agrees with the Reeb flow of the Liouville $1$-form $\l$ (see, for example, \cite[Theorem 1.5.2]{ge}). As we explain in Remark~\ref{rem: error} below, Ishikawa's Lefschetz fibrations can be considered on $DT^*\S$ (rather than  $DT\S$) and thus they induce open books for $ST^*\S$. Therefore, the A'Campo-Ishikawa open book supports  $\xi$ as well when viewed for $ST^*\S$, as a consequence of Theorem~\ref{thm: main}.

\begin{corollary}\label{coro: i}
For each admissible divide $P$ on a closed and oriented \red{Riemannian}  surface $\S$, the contact structure supported by the A'Campo-Ishikawa open book (when viewed for $ST^*\S$) is isotopic to the canonical contact structure $\xi$.
\end{corollary}

Our method of proof also implies that for each connected divide on $D^2$, the open book that A'Campo constructs for $\Sp^3$  supports the standard contact structure, which can also be deduced from the main result of another article of Ishikawa~\cite{i2}. Notice that  A'Campo already proved \cite[Section 4]{aca} that the binding of his open book is transverse to the standard contact structure, which is a necessary condition.

\begin{remark} The A'Campo-Ishikawa open book is independent of the given metric on the surface $\S$, up to isomorphism, and the Giroux open book for $V(\S)$ does not need a metric on $\S$, by its construction. Moreover, when $\S$ is of genus at least two and the metric is hyperbolic, the geodesic flow is independent of the metric, up to homeomorphism, by a theorem of Gromov \cite{gr} (see Section~\ref{sec: Birkhoff} for details) and hence the geodesic open book is uniquely determined up to isomorphism.  More generally, if $\S$ is of any genus and equipped with two different  metrics,  the resulting geodesic open books are {isomorphic} as long as $P$ is {\em convex}  with respect to both metrics.  \end{remark}

Provided that the genus of $\S$ is at least one, none of the open books mentioned in Theorem~\ref{thm: main} are planar, by construction.  
This is consistent with the fact that the support genus of $(ST^*\S, \xi)$ is equal to one for a surface $\S$ of positive genus,  which follows by combining an obstruction to planarity due to Etnyre \cite{et}, with constructions from Birkhoff-Fried (construction (3) in Theorem~\ref{thm: main}~\cite{b, f}), or from Massot \cite{m} when the genus of $\S$ is at least two and Van Horn-Morris \cite{vhm} for a surface $\S$ of genus one (see \cite{oz} for further details). 
The only case for which the open book in Theorem~\ref{thm: main} is planar is the one where the divide $P$ is a single embedded circle on~$S^2$. The corresponding open book adapted to $(ST^*S^2 = \R P^3, \xi)$ is planar, whose page is an annulus and monodromy the square of the Dehn twist along the core circle. Here, we characterize those open books in Theorem~\ref{thm: main} with genus one pages.

\begin{theorem}\label{thm: genusone}
Up to homeomorphism, there are exactly three admissible divides on a closed and oriented surface $\S$ of genus at least one, as depicted in Figure~\ref{F:GenusOne},  that yield genus one open books obtained by any of the methods listed in Theorem~\ref{thm: main}.   These open books have~$4g, 4g+2,$ and $4g+4$ binding components,  respectively.
\end{theorem}

As another consequence of  Theorem~\ref{thm: main}, one can draw the following conclusion.
The monodromy of the Giroux open book, which is not immediately clear from the construction, can be computed via Theorem~\ref{thm: main}, since the monodromy of the A'Campo-Ishikawa open book is given explicitly as a product of Dehn twists~\cite{aca} and the monodromies of  the A'Campo-Ishikawa and the geodesic open books are given as a product of two involutions~\cite{acb, dl}.

\subsection{Pairwise identification of the $3$-manifolds in Theorem~\ref{thm: main}} \label{sec: identify} The identification of the bundle of cooriented lines $V(\S)$ with \red{the unit tangent bundle} $ST\S$ is obtained by taking  the unit tangent vector positively normal to the given cooriented line.  The identification of $V(\S)$  with $ST^*\S$ is obtained by further composing the identification above with the natural bundle map $ST\S \to ST^*\S$ induced by the Riemannian metric on $\S$. Equivalently, a more direct identification of $V(\S)$  with $ST^*\S$ is obtained by taking  for  each cooriented line $L$ in $T_q\S$,  the unit covector (a linear map $T_q \S \to \R$)  in $T_q^*\S$, whose kernel is $L$ with its coorientation. In other words, the bundle of cooriented lines tangent to  $\S$ has a natural identification with the bundle of rays in $T^*\S$, which therefore can be denoted by $ST^*\S$ (see, for example, \cite{ma}), but for the purposes of this paper, we opted to denote the bundle of cooriented lines by $V(\S)$, to distinguish it from the unit cotangent bundle, although they are orientation-preserving  diffeomorphic to each other.

\subsection{Open book constructions in chronological order} Our goal here is to bring together seemingly unrelated work in the literature over a span of  almost one hundred years, going all the way back to Birkhoff's article  \cite{b} on dynamical systems, where he constructs a ``surface of section" for the geodesic flow on the unit tangent bundle $ST\S$. The existence of such a surface reduces  the study of some dynamical properties of the flow to understanding a self-homeomorphism of this surface. Birkhoff's approach was later popularized by Fried \cite{f}, who utilized it to study arbitrary transitive Anosov flows on closed $3$-manifolds. With the terminology used in this paper, the surface of section $S$ is a page of an open book for $ST\S$,  so that the interior of $S$ is transverse to the geodesic flow, while  the binding $\del S$ is  a collection of periodic orbits.

The notion of a divide on $D^2$, and the link of a divide  in $\Sp^3$, was introduced by A'Campo \cite{aca, ac}, who was  interested in computing geometric monodromies of isolated plane curve singularities. In particular, A'Campo proved that the link of any isolated plane curve singularity appears as the link of a  divide, the link of a connected divide is fibered, and that this fibration is a model for the Milnor fibration of the singularity. Note that, for each connected divide, A'Campo's construction gives nothing but an open book for $\Sp^3$  whose binding is the link of the given divide.  The work of A'Campo was later generalized by Ishikawa \cite{i}, who showed that the link of an admissible divide\footnote{Birkhoff used the term \emph{primary set of curves}  for what Ishikawa (following A'Campo) called an {\em admissible divide.}}  on an arbitrary surface $\S$ (not just $D^2$) is fibered in~$ST\S$  and moreover $DT\S$ admits an {\em achiral} Lefschetz fibration over $D^2$,  which restricts to  the same open book for its boundary.

In another direction, at the turn of this century, a major breakthrough in contact topology was achieved by Giroux \cite{gi}, who proved that every closed contact manifold is convex, which is equivalent to the fact that every contact manifold admits an adapted open book with Weinstein pages. By combining with the earlier work of Giroux \cite{g}, it follows in particular that $V(\S) \cong ST^*\S$ admits an open book  \cite{m} that supports its canonical contact structure $\xi$.

In this article, we relate the three different constructions of open books mentioned in the last three paragraphs, where the first one involves  the dynamics of the geodesic flow, the second one is motivated by the study of isolated plane curve singularities, and the third one illustrates the power of convexity in contact topology. 

\begin{remark} \label{rem: johns} Building on the work of Seidel \cite{s}, and inspired by the work of A'Campo \cite{ac},  Johns \cite{jo}  constructed an exact symplectic Lefschetz fibration on the symplectic disk cotangent bundle $(DT^*\S, \o)$, which in turn, restricts to yet another open book for $ST^*\S$.  The difference from Ishikawa's Lefschetz fibration is that Johns' fibration is not just a smooth fibration but adapted to the canonical symplectic structure $\o=d\l$, meaning that each regular fiber is a symplectic submanifold. The main input in Johns' construction of the Lefschetz fibration $DT^*\S \to D^2$ is also a Morse function $f: \S \to \R$. As a matter of fact, based on the handle decomposition of the surface $\S$ given by $f$, Johns describes how to construct (the diffeomorphism type of) the regular fiber of an  exact symplectic Lefschetz fibration $DT^*\S \to D^2$, and  the Lagrangian vanishing cycles explicitly on this fiber. Roughly speaking, Johns' Lefschetz fibration can be viewed as a symplectic convexification of the Morse function $f$, as opposed to a contact convexification due to Giroux, which we discuss in Section~\ref{sec: giroux}. Using methods similar to the ones employed in~\cite{oz}, one can prove that, once a suitable Morse function is fixed on $\S$, the Lefschetz fibrations of Johns and Ishikawa are actually isomorphic.

\end{remark} 

\begin{remark}  It is not true that every open book adapted to $(ST^*\S, \xi)$ arises from an admissible divide or an ordered Morse function on $\S$, using any one of the (eventually equivalent) constructions in Theorem~\ref{thm: main}. This is simply because each open book appearing in Theorem~\ref{thm: main} has an even number of binding components and there is no reason for this to hold for an arbitrary open book adapted to $(ST^*\S, \xi)$. For example, a genus one open book  adapted to~$(ST^*T^2 = T^3, \xi)$ with {\em three} binding components, was constructed by Van Horn-Morris \cite{vhm} (see also \cite[Section~3]{dd} for similar examples obtained by the first author). Note that one can always get an adapted open book with an odd number of binding components, by positively stabilizing an adapted open book with an even number of binding components, but this is not the case for the aforementioned example of Van Horn-Morris.
\end{remark}

\noindent{\bf Outline of the paper:} First, we review in some details the three  distinct constructions above due to A'Campo $\&$ Ishikawa, Giroux, and Birkhoff $\&$ Fried in  Sections~\ref{sec: ACampo} to~\ref{sec: Birkhoff}, respectively.  Then we describe a proof of Theorem~\ref{thm: main} in Section~\ref{sec: proof}. 
Finally we give a proof of Theorem~\ref{thm: genusone} in  Section~\ref{sec: genus}.

\section{Complexification of a Morse function on a surface}\label{sec: ACampo}

In this section, we briefly review Ishikawa's construction \cite{i} of Lefschetz fibrations  on the disk (co)tangent bundle of a closed and oriented  surface $\S$,  which is a generalization of the work of A'Campo \cite{aca, ac}.
These Lefschetz fibrations induce natural open books for the unit (co)tangent bundle by restricting to the boundary.

For any admissible divide $P \subset \S$, there is an ordered Morse function $f : \S \to \mathbb{R}$ adapted to $P$, satisfying the following conditions

$\bullet$ $P = f^{-1}(0)$,

$\bullet$ each double point of $P$ corresponds to a critical point of $f$ of index~$1$, 

$\bullet$ each black (resp. white) region of $\S \setminus P$ contains exactly one index~$2$ (resp.~$0$) critical point of $f$, \red{and}

$\bullet$  \red{$f$ does not have any other critical points.} 

\begin{figure}[h!]
	\includegraphics[width=.49\textwidth]{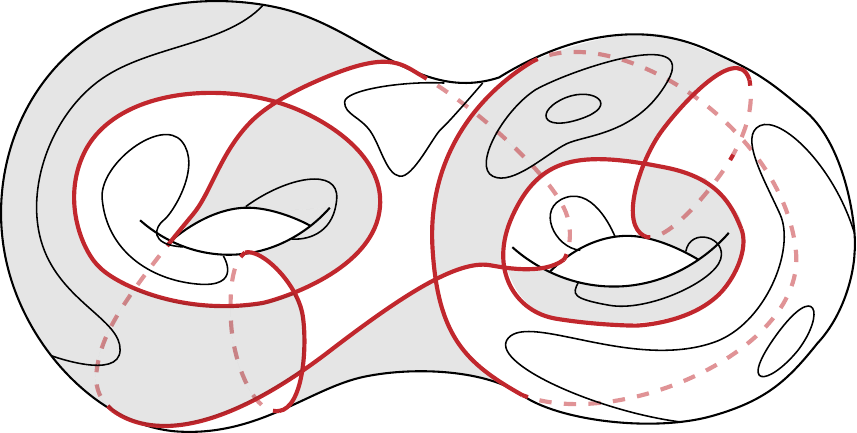}
	\caption{An admissible divide~$P$ (in red) on a genus 2 surface, a black-and-white coloring of the complement~$\S\setminus P$, and some level sets of an ordered Morse function adapted to~$P$. The divide $P$ consists of six curves and includes seven double points.}
	\label{F:Divide}
\end{figure}

Recall that by definition of a Morse function, at every critical point~$p$ of~$f$, one can find a small neighbourhood $U_p$ and local Morse coordinates on~$U_p$, that is, coordinates $(x_1, x_2)$ such that $f(x_1, x_2)=f(p)\pm x_1^2\pm x_2^2$. 
One can then pick a cut-off function $\chi: \S \to \R$ which vanishes outside the union of the~$U_p$ and is equal to $1$ on a smaller neighbourhood of the critical points.

The Morse function $f: \S \to \mathbb{R}$ can be extended to an \emph{almost complexified} Morse function $f_\C : T\S \to \mathbb{C}$ defined as
\begin{equation} \label{eq: map}
f_\C (x,u) = f (x) + i \eta df_x (u) - \dfrac{1}{2} \eta^2 \chi (x) d^2f_x (u,u)
\end{equation}
for $x \in \S$ and $u \in T_x\S$. Here the linear map $df_x :
T_x \S \to \R$ denotes the differential, and the bilinear map $d^2f_ x : T_x\S \times T_x\S \to \R$ denotes the Hessian of $f$ at the point $x$ in the local Morse coordinates.  Also $\eta$ is a sufficiently small positive real number, whose role will be clarified below.

The map $f_\C$ is almost complexified in the sense that the Morse singularities of~$f$ becomes complex Morse singularities of $f_\C$. More precisely, \red{setting $z_j=x_j+i \eta u_j$}, the Morse singularities of $f$ of the form $-x_1^2 -x_2^2, -x_1^2 +x_2^2, x_1^2 +x_2^2$ in local coordinates on $\S$ becomes  singularities of $f_{\C}$ of the form  $-z_1^2 -z_2^2, -z_1^2 +z_2^2, z_1^2 +z_2^2$, respectively, in local complex coordinates on $T^*\S$, which are all equivalent to the complex Morse singularity $z_1^2 + z_2^2$, under a complex change of coordinates. Moreover, by choosing $\eta$ sufficiently small, one can guarantee that there are no other critical points of $f_{\C}$. It follows that  a point in $T\S$  is a critical point of $f_{\C}$  if and only if it belongs to $\S$ and it is a critical point of $f$. The map   $f_{\C} :  T\S \to \mathbb{C}$ descents to an {\em achiral} (see Remark~\ref{rem: error}) Lefschetz fibration $D T\S \to D^2$.  The singular fiber over~$0 \in D^2$ contains a singularity for each double point of $P$.  Corresponding to each index $2$ or index $0$ critical point of $f$, there is a singular fiber  containing a unique singularity. Moreover,  the regular fiber and the vanishing cycles of the Lefschetz fibration $D T \S \to D^2$, are described explicitly in \cite{i} by a method due to A'Campo \cite{ac}.

\begin{remark}\label{rem: error} There is an orientation issue in \cite[Lemma 2.6]{i}. Following the notation in \cite{i}, suppose that $x_1, x_2$ are local coordinates on the surface $\S$ and $u_1, u_2$ are the corresponding coordinates in the {\em tangent} fibers, so that $\{x_1, x_2, u_1, u_2\}$ is an oriented chart for $T\S$. But the orientation of this chart is opposite to the complex orientation that is required in the definition of a Lefschetz fibration, since Ishikawa uses the complex chart  $$(z_1, z_2)= \red{(x_1+i\eta u_1, x_2+i\eta u_2)}$$ for $T\S$ in the proof of \cite[Lemma 2.6]{i}. Except the   orientation issue, however, everything else works in Lemma 2.6 and the subsequent discussion leading to Proposition 3.1 and Lemma 3.2 in \cite{i}.  In other words, the Lefschetz fibration Ishikawa constructs is {\em achiral} on $DT\S$ and the corresponding Lefschetz fibration $D T^* \S \to D^2$, obtained by reversing the orientation,  induces an open book on the boundary $S T^*\S$. It follows, by Theorem~\ref{thm: main}, that  this open book supports the contact canonical structure on $S T^*\S$, which is not clear from Ishikawa's work \cite{i}.  \end{remark}

\begin{figure}[h!]
	\begin{picture}(360,395)(0,0)
	\put(0,300){\includegraphics[width=.38\textwidth]{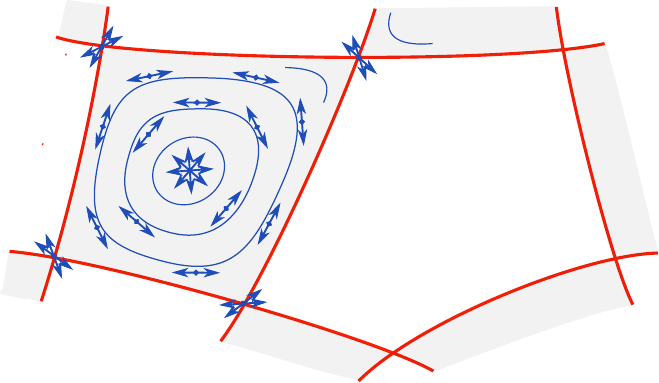}}
	\put(90,350){$S_0$}
	\put(180,300){\includegraphics[width=.38\textwidth]{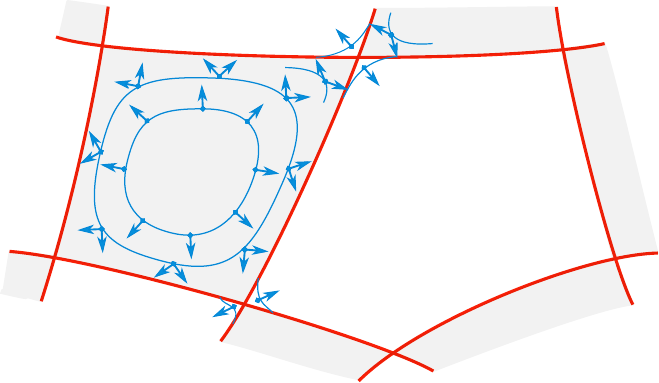}}	
	\put(270,350){$S_{\frac\pi4}$}
	\put(0,200){\includegraphics[width=.38\textwidth]{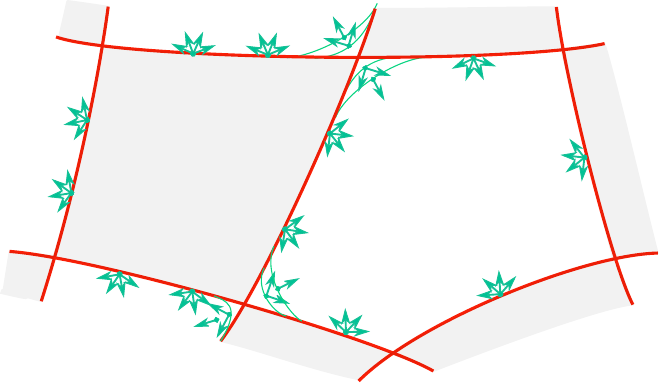}}
	\put(50,250){$S_{\frac\pi2}$}
	\put(180,200){\includegraphics[width=.38\textwidth]{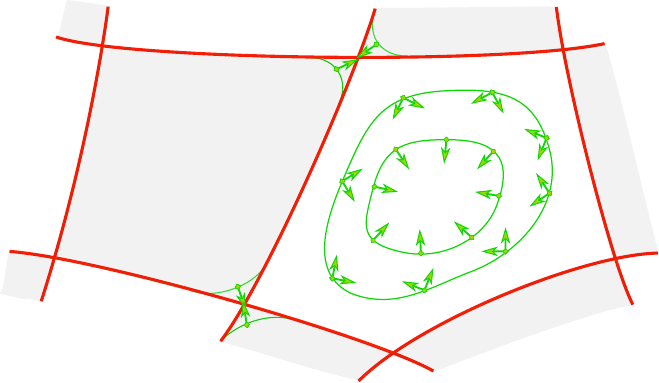}}
	\put(230,250){$S_{\frac{3\pi}4}$}
	\put(0,100){\includegraphics[width=.38\textwidth]{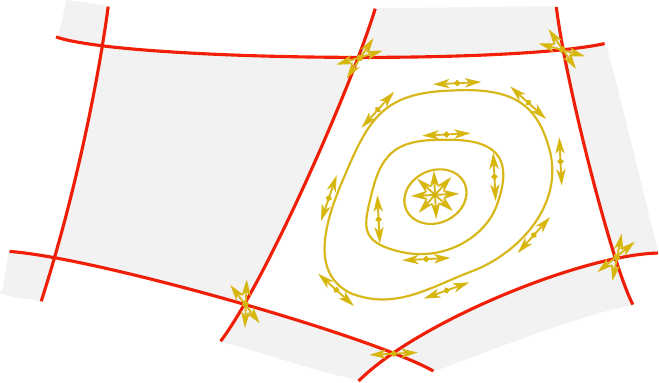}}
	\put(50,150){$S_\pi$}
	\put(180,100){\includegraphics[width=.38\textwidth]{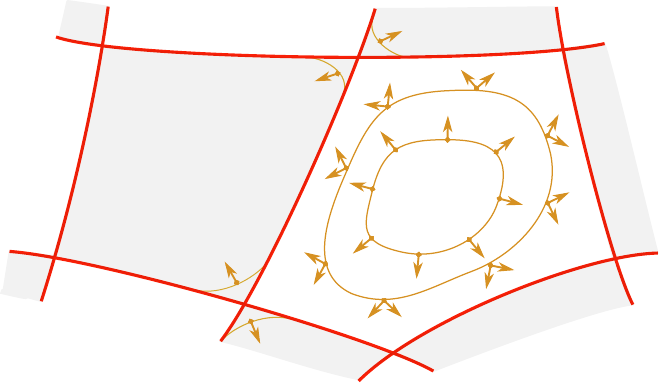}}
	\put(230,150){$S_{\frac{5\pi}4}$}
	\put(0,0){\includegraphics[width=.38\textwidth]{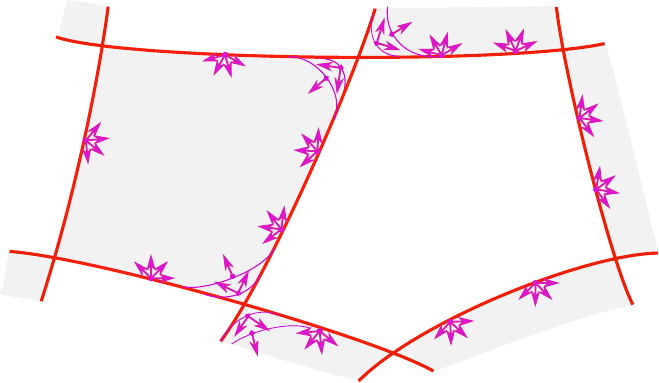}}
	\put(90,50){$S_{\frac{3\pi}2}$}
	\put(180,0){\includegraphics[width=.38\textwidth]{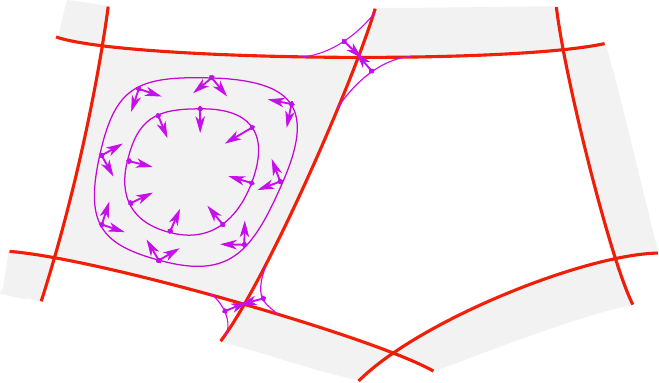}}
	\put(270,50){$S_{\frac{7\pi}4}$}
	\end{picture}
	\caption{Some of the pages  $S_\theta$ of the A'Campo-Ishikawa open book for~$ST\S$, viewed as sets of vectors tangent to~$\S$.}
	\label{F:ACampo1}
\end{figure}
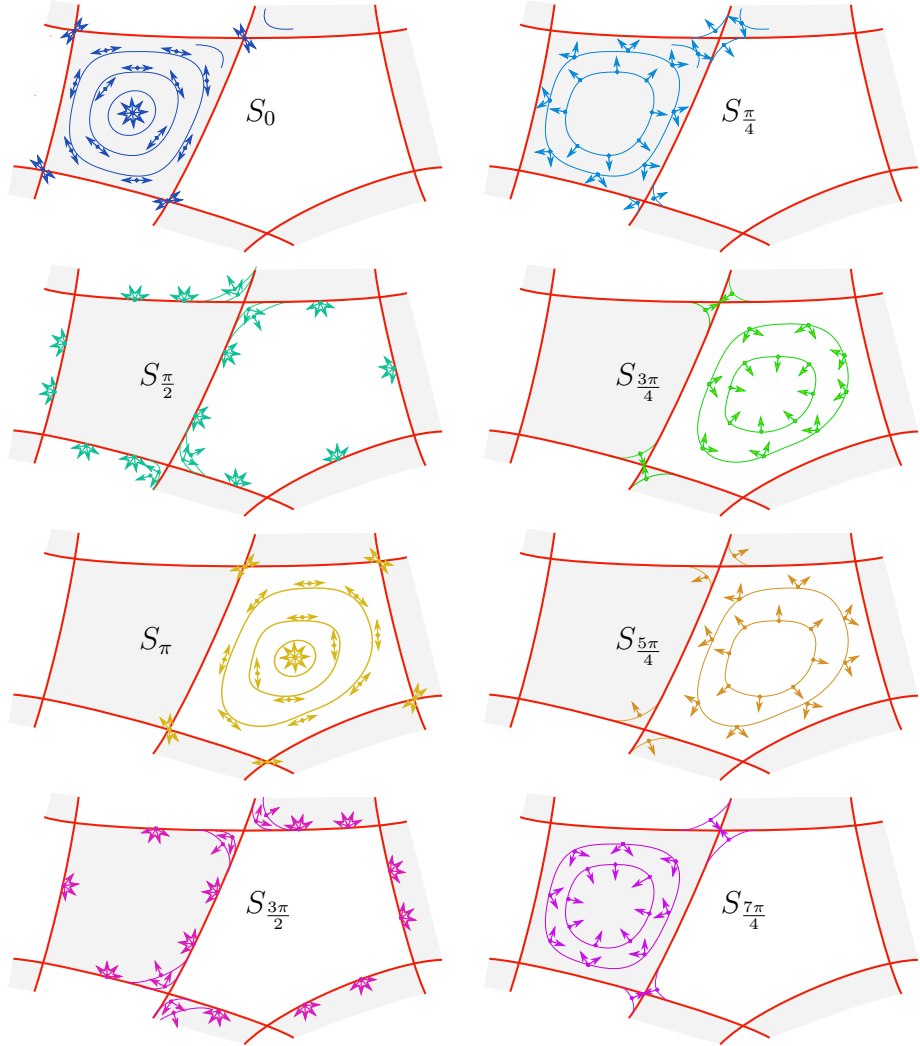

When restricted to $ST\S$, the complexification map $f_\C$ in~\eqref{eq: map} vanishes on the link $L(P)$ of the divide $P$, which is the set of unit vectors tangent to $P$ \cite{ac, i}.

\begin{theorem}[A'Campo-Ishikawa]\label{thm: ACampo}
\red{Provided that $\eta$ is sufficiently small, the map $f_\C$ restricts to a fibration $ST\S \setminus L(P) \to S^1$, obtained by dividing the image by its norm, and therefore inducing an open book for  $ST\S$ with binding $L(P)$.}
\end{theorem}

This is the \emph{A'Campo-Ishikawa}  open book mentioned in Theorem~\ref{thm: main}. If the unit circle $S^1$ is parametrized by $\theta$, then we denote the closure of the inverse image of $\theta$ under the fibration map as $S_\t$, the $\theta$-page.

As depicted in Figure~\ref{F:ACampo1}, the interior of $S_0$ (resp. $S_\pi$) consists of the unit vectors tangent to the level sets of $f$ in black (resp. white) regions, while the interior of~$S_{\pi/2}$ (resp. $S_{3\pi/2}$) consists of the unit vectors along the divide $P$ pointing into black (resp.  white) regions of $\S \setminus P$, except around the double points of the divide where the term involving the Hessian in Equation~\eqref{eq: map} smooths the surface and separate the surfaces associated to different values of~$\theta$.
Around the double points of~$P$, where the function~$\chi$ does not vanish, the surfaces~$S_\theta$ are depicted in Figure~\ref{F:ACampo2}.

\begin{figure}[h!]
	\includegraphics[width=.22\textwidth]{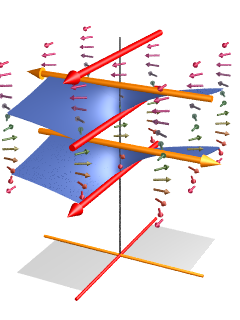}
	\includegraphics[width=.22\textwidth]{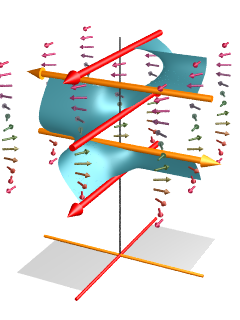}
	\includegraphics[width=.22\textwidth]{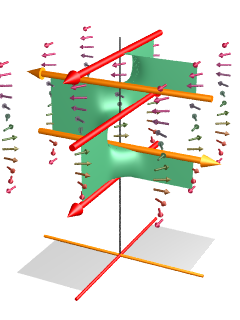}	
	\includegraphics[width=.22\textwidth]{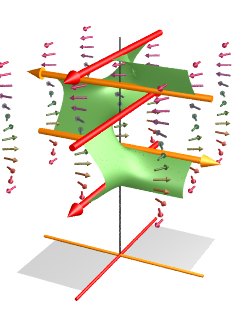}
	
	\includegraphics[width=.22\textwidth]{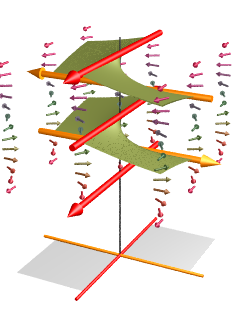}
	\includegraphics[width=.22\textwidth]{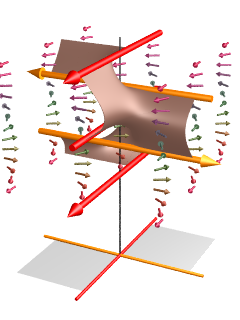}
	\includegraphics[width=.22\textwidth]{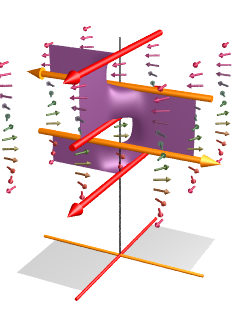}
	\includegraphics[width=.22\textwidth]{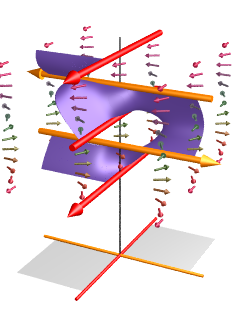}
	\caption{The pages $S_\theta$ of the A'Campo-Ishikawa  open book for~$ST\S$ around the fiber of a double point of a divide~$P$, for $\theta=\frac{k\pi}4$ with $k=0, \dots, 7$. \red{The arrows correspond to the generators of the geodesic flow on~$ST\S$, in the case where the divide~$P$ is geodesic. They are tangent to the link~$L(P)$ and transverse to the interiors of all pages.} }
	\label{F:ACampo2}
\end{figure}

The topological type of the pages of the A'Campo-Ishikawa  open book can be easily determined as follows. We denote by $v(P)$ the number of double points of a divide~$P$ and by~$c(P)$ the number of circles that form~$P$. In Figure~\ref{F:Divide}, for example, \red{$c(p)=6$ and $v(p)=7$. }

\begin{proposition}\label{prop: genus}
For any admissible divide~$P$, the A'Campo-Ishikawa open book has~$2c(P)$ binding components and its page genus is given by~$1+v(P)-c(P)$.
\end{proposition}

\begin{proof} Since all the pages have the same genus, we focus on one page, say~$S_{\frac\pi2}$, in particular. 
First notice the boundary of~$S_{\frac\pi2}$ is the link~$L(P)$ of the divide, which has two components for every circle in~$P$. Now consider $P\subset\S$ as a graph whose vertices are the double points of~$P$ and whose edges are the segments of~$P$ connecting these double points.
Then the surface $S_{\frac\pi2} \subset ST\S$ is made of one rectangle in the fiber of every segment of~$P$.
The horizontal sides of these rectangles are in the link~$L(P)$, hence in the boundary of~$S_\frac{\pi}2$.
The vertical sides can be divided into two segments, which are glued with the rectangle associated to the next and previous edge of~$P$ respectively, when rotating around the corresponding vertex of~$P$, as depicted in Figure~\ref{F:Rectangles}.
One checks that every rectangle contributes by~$-1$ to the Euler characteristics of~$S_{\frac\pi2}$.
Since $P$ has twice more edges than vertices, we deduce that $\chi(S_{\frac\pi2})=-2v(P)$.
The number of boundary components of~$S_{\frac\pi2}$ is $2c(P)$.
Hence the genus of~$S_{\frac\pi2}$ is $1+v(P)-c(P)$.
\end{proof}

\begin{figure}[h!]
	\includegraphics[width=.3\textwidth]{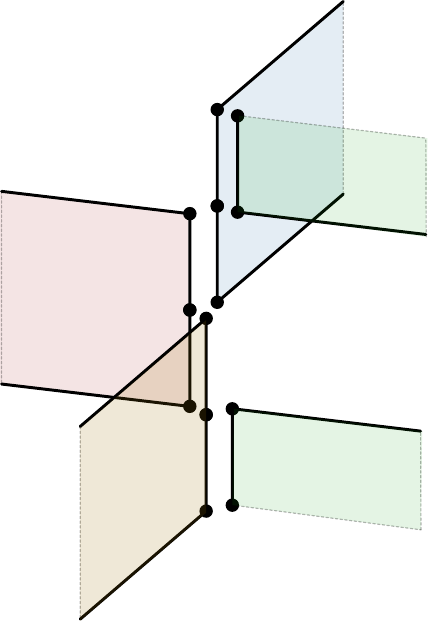} 
    	\caption{A decomposition of the page~$S_{\frac\pi2}$ in the A'Campo-Ishikawa's open book into rectangles, each of which corresponds to one edge of the divide. The figure shows the four rectangles adjacent to the fiber of a double point of the divide. Each rectangle consists of one face, 6 edges, and 6 vertices. The horizontal edges (2 per rectangle) correspond to the link of the divide and form the boundary of~$S_{\frac\pi2}$. The vertical edges (4 per rectangle) are glued 2 by 2. The vertices are glued 3 by 3. Hence the contribution of each rectangle to the Euler characteristics of~$S_{\frac\pi2}$ is $1-2-4\cdot\frac12+6\cdot\frac13=-1$.}
	\label{F:Rectangles}
\end{figure}

\section{Convex Morse functions on the bundle of hyperplanes} \label{sec: giroux}

Our goal in this section is to review Giroux's construction~\cite{g} of an open book adapted to the contact $3$-manifold~$(V(\S), \xi)$, where $\xi$ is the canonical contact structure on the bundle~$V(\S)$ of cooriented lines tangent to $\S$.
We begin with some general facts about convex contact manifolds, and then give some details of Giroux's proof of the convexity of the canonical contact structure on the bundle of cooriented hyperplanes~$V(M)$ for an arbitrary smooth manifold~$M^{n+1}$ in Section~\ref{sec: conv}, before we turn our attention to the case $n=1$ (i.e, $M=\S$, a surface) in Section~\ref{sec: cano3}.

\begin{definition} \cite{g} A contact structure $\zeta$ is called {\em convex}  if it is invariant under the flow of a vector field $X$ which is gradient-like for a Morse function.  Such a function is called $\zeta$-convex (or {\em contact}) Morse, and $X$ is called a {\em contact} vector field.  The {\em characteristic hypersurface} $C_X$ is defined as the set of points where $X$ is tangent to  $\zeta$.
 \end{definition}

In his groundbreaking work, Giroux \cite{gi} proved that every contact manifold admits an adapted open book. This result follows as a corollary of  a much harder theorem of Giroux which says that every contact structure is convex, a proof of which can also be found in Sackel's Ph.D. thesis \cite{sa}.  
We formulated the two celebrated results of Giroux as Theorem~\ref{thm: gir} below.

\begin{theorem}[Giroux]\label{thm: gir} Let $(V, \zeta) $ be a closed contact manifold of dimension $2n+1$. 
Suppose that $X$ is a contact vector  field on $V$, which is gradient-like for an ordered Morse function $F: V \to \R$. 
Let $L$ be a regular level set of $F$ above the critical values of index $n$ and below the critical values of index $n+1$. 
Then there is an open book adapted to $(V, \zeta) $ whose binding $K$ is the transverse intersection of  $L$ and the characteristic hypersurface~$C_X$. 
Moreover, the binding $K= L \cap C_X$ cuts $L \cup C_X$ into four pages of this open book.

Conversely, any open book for $V$ supporting $\xi$ comes from this construction.
\end{theorem}

\subsection{Convexity of the canonical contact structure on the bundle of hyperplanes } \label{sec: conv}

Let $V(M) $ denote the  bundle of cooriented hyperplanes tangent to a closed manifold $M^{n+1}$ and $\xi$ denote the canonical contact structure on $V(M)$. Note that $V(M)$ can be identified with the unit cotangent bundle $S T^*M$  and $\xi$ is given  by the kernel of the Liouville $1$-form $\l$, under this identification (see, for example, \cite[page 32]{ge}).   In the following, we review Giroux's construction \cite[Example 4.8]{g} of a $\xi$-convex Morse function on $V(M)$.

Any diffeomorphism $\psi: M \to M$ lifts to a contactomorphism $\widetilde{\psi}$ of $(V(M), \xi)$ defined by $$\widetilde{\psi} (x, H) = (\psi (x), d\psi_x (H)).$$
It follows that any vector field $X$ on $M$ lifts to a contact vector field $\widehat{X}$  on $(V(M), \xi)$ by lifting the flow of $X$.

Let $f: M \to \R$ be a Morse function and $X$ be a gradient-like vector field for $f$ such that at each critical point of $p$ of $f$, the eigenvalues of $D_pX$ are {\em real} and {\em simple}.
Then there is a contact vector field $\widehat{X}$ for $(V(M), \xi)$ obtained as the lift of $X$ as explained above.
To show that $(V(M), \xi)$ is convex, it suffices to construct a Morse function $f_\xi : V(M) \to \R$ such that  $\widehat{X}$ is gradient-like for $f_\xi$. 
For any critical point~$p$ of $f$, let $V_p (M) := \pi^{-1} (p) $ denote the sphere fiber, where $\pi : V(M) \to M$ is the bundle projection.  Since $\widehat{X}$ projects to $X$, it is vertical above the critical points of $f$, i.e., it is tangent to $V_p$. The restriction of $\widehat{X}$ to $V_p(M)$ is the gradient of some Morse function $g_p : V_p(M) \to \R$, having exactly $2n+2$ critical points, corresponding to the $n+1$ hyperplanes generated by~$n$ eigendirections of~$D_pX$.

We set
\begin{equation}\label{eq: fxi}
f_\xi := f + \sum\limits_{p \in Crit(f)} \chi_p g_p,
\end{equation}
where $\chi_p$ is a suitably defined cut-off function which vanishes outside of a sufficiently small neighborhood of the critical points.

\begin{figure}[h!]
	\includegraphics[width=.42\textwidth]{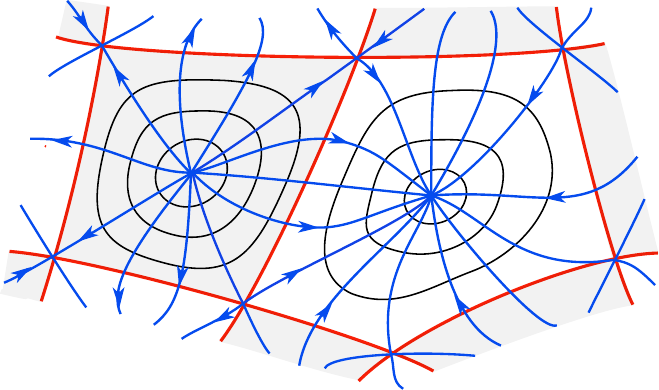}
	\caption{Flow lines of a gradient-like vector field for an ordered Morse function associated to a divide on a surface.}
	\label{F:DivideConvexe}
\end{figure}

\subsection{Canonical contact structure on the bundle of cooriented lines} \label{sec: cano3}
We restrict our attention to the case $n=1$, namely we take $M$ to be a closed and oriented  surface $\S$. Suppose that~$f: \S \to \R$  is an ordered Morse function and~$X$ is  a gradient-like vector field for $f$ on~$\S$ (see Figure~\ref{F:DivideConvexe} for an example).  The vector field $X$ lifts to a vector field~$\widehat X$ on~$V(\S)$, which is gradient-like for the modified lift~$f_\xi:V(\S)\to\R$. As observed by Massot \cite{m}, the $\xi$-convex Morse function $f_\xi : V(\S) \to \R$ defined by Equation~\eqref{eq: fxi} is {\em ordered} provided that $f: \S \to \mathbb{R}$ is an {\em ordered}  Morse function\footnote{Massot \cite{m} uses a self-indexed Morse function $f$, but the conclusion holds true if self-indexed is replaced by ordered. Here we assume that $f$ is ordered and $f^{-1}(0)$  includes all index $1$ critical points.}.
In particular the singularities of~$f$ lift to quadruples of singularities of~$f_\xi$ in~$V(\S)$, as for example,  we depicted on the left in  Figure~\ref{F:CriticalPoint} for an index~$1$ singularity that lifts to two index~$1$ and two index ~$2$ singularities.

\begin{figure}[h!]
	\begin{picture}(420,200)(0,0)
	\put(0,-15){	\includegraphics[width=.35\textwidth]{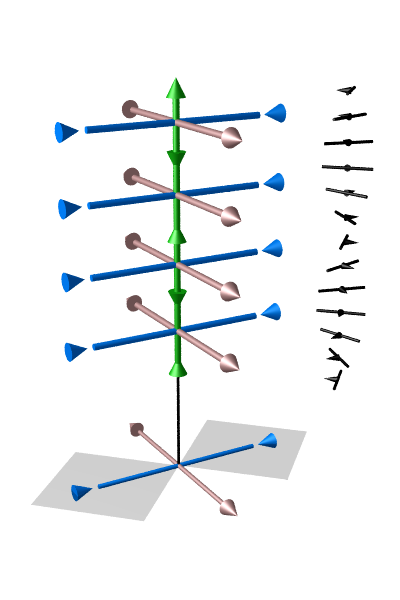}}
	\put(140,-15){	\includegraphics[width=.35\textwidth]{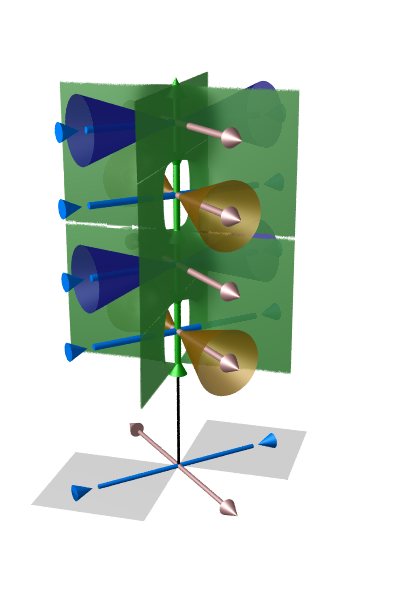}}
	\put(275,-15){	\includegraphics[width=.35\textwidth]{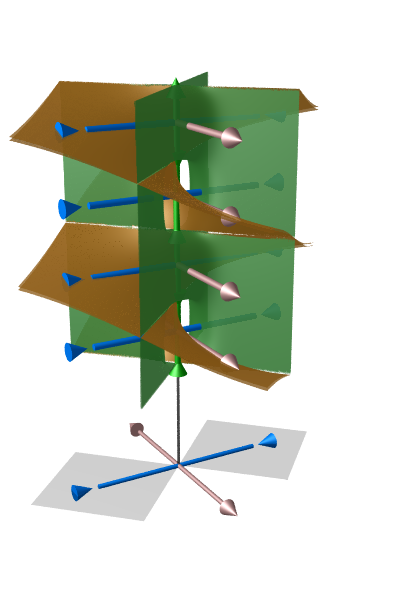}}
	\end{picture}
	\caption{Given a vector field~$X$ on a surface~$\S$ and a hyperbolic fixed point~$p$, the differential of the flow shrinks tangent lines toward the expanding eigendirection. The lift~$\widehat{X}$ in the bundle~$V(\Sigma)$ has four critical points in the fiber above~$p$, which correspond to the cooriented lines tangent to the eigendirections as depicted on the left.
	When $X$ is a (pseudo-) gradient of an ordered Morse function~$f$ on~$\S$, one can lift~$f$ to~$V(\S)$ and modify it using Equation~\eqref{eq: fxi} into an ordered Morse function~$f_\xi$, so that the critical points of $f_\xi$ are the critical points of $\widehat{X}$ and the levels of $f_\xi$ away from the fibers of the critical points of~$f$ are the lifts of the levels of~$f$. In the center we depicted two critical level sets (blue and yellow) and the intermediate regular level set (green)~$f_\xi^{-1}(0)$. On the right, we depicted the level set~$f_\xi^{-1}(0)$ (green) and the characteristic surface~$C_{\widehat X}$ (yellow). \red{The gray squares at the bottom represent the black components of the complement of the divide.} } 
	\label{F:CriticalPoint}
\end{figure}

Now one can apply Theorem~\ref{thm: gir} to find an open book, which we call a \emph{Giroux}  open book,  that supports the canonical contact structure $\xi$ on $V(\S)$ using the $\xi$-convex ordered Morse function $f_\xi : V(\S) \to \R$.
On the right in Figure~\ref{F:CriticalPoint}, we depicted  the characteristic surface~$C_{\widehat X}$ (yellow) and the regular level~$f_\xi^{-1}(0)$ (green) around the fiber of a double point of the divide.
They intersect along the link consisting of the cooriented lines based at~$P$ and tangent to~\red{$X$.}  We provide some more details of Giroux  open book in the proof of Lemma~\ref{lem: bin}. 

Suppose that $P$ is an admissible divide on $\S$ and let $f: \S \to \R$ be an ordered Morse function adapted to $P$.
Recall that~$c(P)$ is the number of circles in~$P$ and~$v(P)$ is the number of double points of~$P$.
By our assumption, $v(P)$ is equal to the number of index~$1$ critical points of $f$.

\begin{proposition}\cite{m}\label{prop: girgenus}
For any admissible divide~$P$, the Giroux open book has~$2c(P)$ binding components and its page genus is given by~$1+v(P)-c(P)$.
\end{proposition}

\begin{proof} The indices of the critical points of $f_\xi: V(\S) \to \R$  can be computed from the indices of the critical points of $f: \S \to \R$, which in turn, determines the topology of the characteristic surface $C_{\widehat{X}}$.
This is because $f_\xi |_{C_{\widehat{X}}}$ is a  Morse function whose critical points are exactly the critical points of $f_\xi$ and moreover an index $i$ critical point of~$f_\xi$ gives a critical point of $f_\xi|_{C_{\widehat{X}}}$, whose index is $i$ if $i \leq 1$, and $i-1$, otherwise (see \cite{g} for this index calculation). Since~$C_{\widehat{X}}$ is the union of the two pages of an open book supporting $\xi$, the topology of the  page is determined by the number of  connected components of the binding, which is equal to twice the number of circles in  $P= f^{-1}(0)$. Since $v(P)$ is equal to the number of critical points of index one of the Morse function~$f$, by our assumption, the Euler characteristic of the page is given by $-2v(P)$, and thus the genus of the page is given by $1+v(P)-c(P)$. \end{proof}

{It follows that the topology of the page of the Giroux open book agrees with that of the A'Campo-Ishikawa open book (see Proposition~\ref{prop: genus}). }

\section{Geodesic flow and Birkhoff sections} \label{sec: Birkhoff}

In this section we recall Birkhoff's classical construction~\cite{b} of a negative Birkhoff cross section for the geodesic flow from a {convex divide (which he called a primary set of curves) on a surface.}  The geodesic flow $\Phi$ is the flow  on $ST\S$, whose orbits are the lifts of the geodesics on $\S$. More precisely, if  $g$ is  a geodesic with unit speed on the Riemann surface~$\S$, then  the orbit of $\Phi$ going through $(g(0), \dot{g} (0))$
is given by $$\Phi^t (g(0), \dot{g}(0)) = (g(t), \dot{g}(t)).$$ Although the geodesic flow depends on the choice of a Riemannian metric on~$\S$, Gromov \cite{gr} showed  that the geodesic flows corresponding to two negatively
curved metrics on a surface are  topologically conjugated, which means that there is a homeomorphism
between the unit tangent bundles of $\S$ with respect to these metrics, taking the oriented orbits of one onto the other.  We conclude that  as far as the topological (rather than dynamical) properties are concerned, there is essentially a unique geodesic flow on a negatively curved surface.

\begin{definition}  Suppose that $X$ is a non-singular vector field on a closed and oriented $3$-manifold $M$.
A {\em Birkhoff cross section}\footnote{also called Poincar\'e-Birkhoff section, or global cross section}  for $(M, X)$  is a compact orientable surface $S$ with boundary such that $S$ is embedded in $M$, $X$  is transverse to the interior of $S$, the boundary $\partial S$ is tangent to $X$, and every orbit of $X$ intersects $S$ after a bounded time  (see Figure~\ref{F:BirkhoffSection} left). \end{definition}

It follows that  $\partial S$ is the union of finitely many periodic orbits
of $X$, and  near its boundary, a Birkhoff cross section $S$ looks like a {\em helicoidal staircase} (see Figure~\ref{F:BirkhoffSection} right).
Notice that~$S$ is cooriented by $X$ since the interior of $S$ is transverse to $X$.
Therefore, there is an induced orientation on $S$ (and on its boundary $\partial S$), since $M$ is oriented.  On
the other hand,  $\partial S$ has a natural orientation as a collection of periodic orbits of the vector field $X$.
 A  Birkhoff cross section $S$ is said to be
a {\em positive} (resp. {\em negative}) if for each component of $\partial S$, the natural orientation given by $X$ coincides with (resp. is opposite of)  its orientation inherited as the boundary of $S$.

\begin{figure}[h!]
	\includegraphics[width=.42\textwidth]{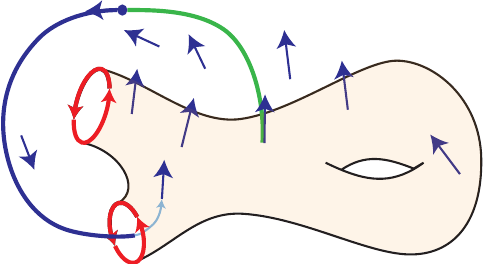}\qquad\qquad\qquad
	\includegraphics[width=.28\textwidth]{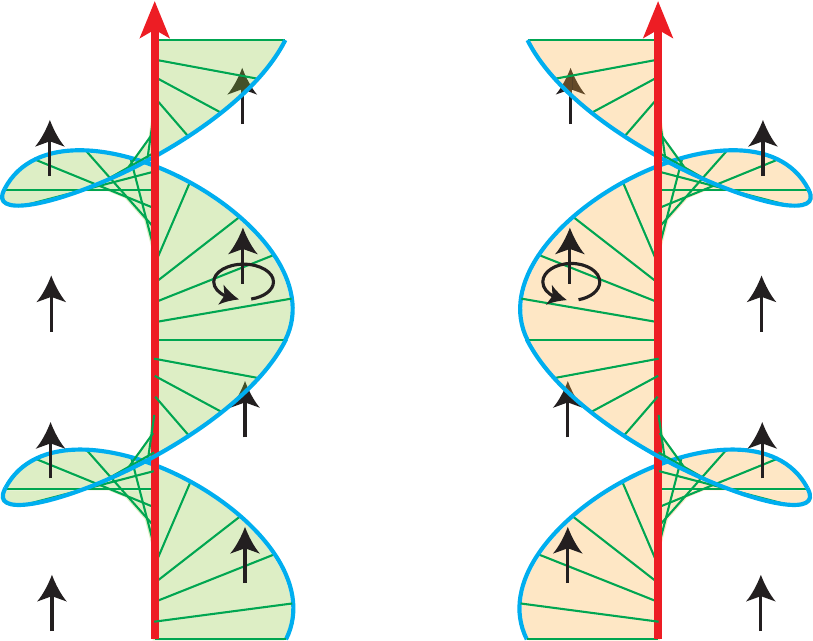}
	\caption{Birkhoff cross sections: on the left the general picture, and on the right how it looks around boundary components, assuming the flow is locally vertical, for a negative and a positive Birkhoff cross section,  respectively.}
	\label{F:BirkhoffSection}
\end{figure}

Since $X$ fixes the coorientation of $S$ and the
orientation of $\partial S$, changing the orientation of $M$ while fixing $X$
changes the orientation of $S$ but not the orientation of $\partial S$.

Notice that a positive Birkhoff cross section $S$ for $(M, X)$ is a page of an open book for $M$ whose oriented binding is $\partial S$.
On the other hand, if  $S$ is a negative Birkhoff cross section  for $(M, X)$, then $\overline{S}$ is a positive Birkhoff cross section
for $(\overline{M}, X)$ and hence $\overline{S}$   is a page of an open book for $\overline{M}$ whose oriented binding is $\overline{\partial{S}}$.

\begin{definition}\label{def: convexedivide}
We say that an admissible divide~$P$ on a surface~$\S$ equipped with a Riemannian metric is~{\em{convex}} if every curve in~$P$ is a closed geodesic, every geodesic on~$\S$ intersects $P$ in bounded time, and every region of~$\S\setminus P$ can be foliated by concentric curves with non-vanishing curvature. In this context, a foliation of the regions of~$\S\setminus P$ by concentric curves is called a {\em{convex}} foliation. \end{definition}

Note that every admissible divide $P$ on a surface $\S$ with constant curvature is convex. 
Loosely speaking, this can be seen as follows: Divide every region of $\S \setminus P$ into triangles by picking a point inside every region, and connecting it to the double points of the divide as depicted on the left in Figure~\ref{F:Bubble}. It remains to foliate every such triangle by a collection of curves that connect the two sides of each triangle that does not belong to $P$ such that they are orthogonal to these sides, and strictly convex. 
This last step can be achieved  by starting with segments that are parallel to the divide and bending them until they become orthogonal to the sides.
Note that the leaves of the foliation we just described can then be chosen as level sets of an ordered  Morse function on $\S$ adapted to $P$. 

An admissible divide may fail to be convex if, for example, there is a closed geodesic staying in one region of~$\S\setminus P$, as depicted on the right in Figure~\ref{F:Bubble}. 

\begin{figure}[ht]
\includegraphics[width=.4\textwidth]{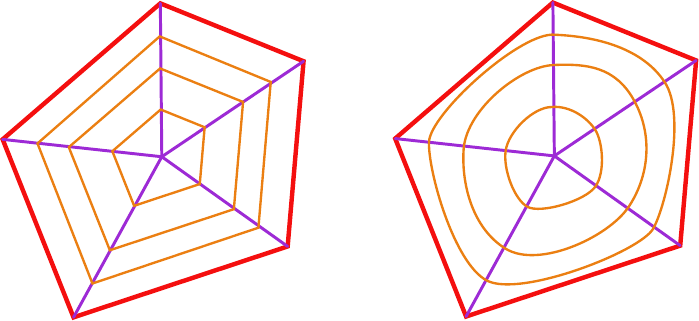}
\hspace{2cm}
\includegraphics[width=.3\textwidth]{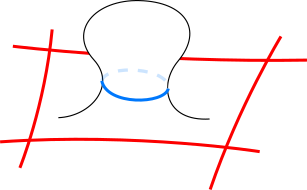} 
\caption{On the left, foliation in  the interior of a geodesic divide (red) in constant curvature by convex curves (orange). On the right, a divide (red) that is not convex: the geodesic (blue) does not intersect the divide.}
\label{F:Bubble}
\end{figure}

\begin{theorem}[\cite{b,f}] \label{thm: bir}
Assume that~$P$ is a convex divide and $\mathcal{F}$ is a convex foliation of~$\S\setminus P$.
The set of pairs $(p,v)$, where $p$ belongs to some black region of  $\S\setminus P$ and $v$ is a unit vector tangent to $\mathcal{F}$ is a Birkhoff cross section for the geodesic flow on~$ST\S$.
\end{theorem}

The proof of Theorem~\ref{thm: bir} is straightforward \red{as we briefly outline here}.  Denoting by~$S^\bullet$ the set of tangent vectors mentioned in Theorem~\ref{thm: bir}, it is easily seen that~$S^\bullet$ is  a surface whose boundary is the link~$L(P)$ of the divide~$P$.
Hence the boundary $\partial S^\bullet$ is indeed tangent to the geodesic flow.
Since the foliation~$\mathcal F$ is convex, the contact between leaves of this foliation and geodesic is only of order 1 (because of the assumption on the curvature), hence of order 0 in the tangent bundle.
This means that the interior of~$S^\bullet$ is transverse to the geodesic flow.
Finally since every geodesic on~$\S$ intersects $P$ in bounded time, every orbit of the geodesic flow intersects~$S^\bullet$ in bounded time.

\begin{remark}
Theorem~\ref{thm: bir} provides one Birkhoff cross section canonically associated to a given divide~$P$, which corresponds to the page~$S_0$ depicted in Figure~\ref{F:ACampo1}.
Another natural Birkhoff cross section is the set of pairs $(p,v)$, where $p$ belongs to some white region of  $\S\setminus P$ and $v$ is a unit vector tangent to $\mathcal{F}$.
It corresponds to the page~$S_\pi$, also   depicted in Figure~\ref{F:ACampo1}.
\end{remark}

\begin{remark} Isotopy classes of {\em negative} Birkhoff cross sections whose boundary is symmetric were classified by Cossarini-Dehornoy~\cite{cd}.
Marty proved that a flow cannot admit at the same time positive and negative Birkhoff cross sections~\cite{marty}, hence there are no {\em positive} Birkhoff cross sections for geodesic flows.
 \end{remark}

\section{Equivalence of open books} \label{sec: proof}

The proof of Theorem~\ref{thm: main} follows by combining Proposition~\ref{prop: onepageisotopic}, Proposition~\ref{prop: 13} and Lemma~\ref{lem: 14} below.

\subsection{A criterion for the isotopy of open books}

In this subsection, we give a proof of Proposition~\ref{prop: onepageisotopic} from the Introduction, which gives a simple criterion to show that two open books on a given $3$-manifold are isotopic. 

\smallskip

\noindent {\bf Proposition 1.2.}
\red{Two open books $(B, \pi)$ and $(B', \pi')$ for a closed and oriented $3$-manifold $M$ are isotopic,  provided that the 
 pages $\overline{\pi^{-1}(0)}$ and $\overline{\pi'^{-1}(0)}$ are isotopic in~$M$. }

\begin{proof} Denote by $S$ and $S'$ the pages $\overline{\pi^{-1}(0)}$ and $\overline{\pi'^{-1}(0)}$, respectively. 
Let $(\phi_t)_{t\in\R}$ denote the isotopy that takes $S$ to $S'$, i.e.,  $\phi_0=\id_M$ and $\phi_1(S)=S'$. 
Note that the open book $(\phi_1(B),\pi\circ\phi_1^{-1})$, which is the image of $(B, \pi)$ under $\phi_1$,  
 satisfies $(\pi\circ\phi_1^{-1})^{-1}(0)=(\phi_1\circ\pi^{-1})(0)=\phi_1(S)=S'$ and it is isotopic to $(B, \pi)$ by $(\phi_t)$. 
Therefore, up to composing by~$\phi_1$, we can assume that $S=S'$ and in particular $B=B'$. 
In order to make the notation lighter, we then denote $\pi:=\pi\circ\phi_1^{-1}$. 
In other words, we can assume that the two open books have a common page and the same binding.

By cutting along $S$, we obtain a manifold~$M\coupe S$ (which is the metric completion of~$M\setminus S$, with two copies of~$S$ in its boundary each corresponding to one face of~$S$), and two functions $\pi:M\coupe S\to[0,1]$ and $\pi':M\coupe S\to[0,1]$, with $\pi|_{S\times\{0\}}=\pi'|_{S\times\{0\}}=0$ and $\pi|_{S\times\{1\}}=\pi'|_{S\times\{1\}}=1$.
Thanks to the normal form of open books along the binding, we can identify $M\coupe S$ with $S\times[0,1]$, so that $\pi$ is the height function $\pi(x,z)=z$. 
In the same vein, we can suppose that $\pi'$ coincides with $\pi$ on~$\partial S\times[0,1]$, and so on all of~$\partial (S \times[0,1])$.

All in all, we have a function $\pi':S\times[0,1]\to[0,1]$, which coincides with the height function of all of $\partial(S\times[0,1])$ and we wonder whether $\pi'$ is isotopic to the height function.
Both functions $\pi$ and $\pi'$ define foliations with trivial holonomy on~$S\times [0,1]$. 
So we can apply the main result of Roussarie~\cite{ro} and deduce that these foliations are isotopic. 
This isotopy realizes an isotopy between the two open books.
\end{proof}

\subsection{\red{Comparing particular pages}}

\begin{definition} \label{def: openb} Suppose that $P$ is an admissible divide on a closed and oriented surface $\S$ which is equipped with a Riemannian metric and  let $f: \S \to \R$ be an ordered Morse function adapted to $P$. Then \red{let $\overline{\ob}_1$ (resp. $\ob_1$) denote the A'Campo-Ishikawa open book for $ST\S$ (resp. $ST^*\S$),  and $\ob_2$ denote the Giroux open book for $V(\S)$. Suppose further that $P$ is {\em convex} with respect to the metric on $\S$.  Let  $\overline{\ob}_3$ denote the geodesic open book for $ST\S$ induced by the geodesic flow, whose pages are {\em negative} Birkhoff cross sections,  and  let $\ob_3$  denote  the corresponding \red{cogeodesic} open book for $ST^*\S$, whose pages are {\em positive} Birkhoff cross sections.}
\end{definition}

\begin{proposition} \label{prop: 13}   Under the identification of $ST\S$ with $V(\S)$, the \red{A'Campo-Ishikawa} open book $\ob_1$ for $ST^*\S$ and the \red{Giroux} open book $\ob_2$ for  $V(\S)$
 are \red{isotopic}.
\end{proposition}

\begin{proof} This is essentially contained in the first author's PhD thesis \cite[Chapter 1]{d}, where tangent bundles were considered rather than cotangent bundles. Here we work with $V(\S) \cong ST^*\S$ which is more natural from the contact geometric point of view. We would like to show that  $\ob_1$ for $ST^*\S$ is orientation-preserving isomorphic to   $\ob_2$ for $V(\S)$.  Instead, we first set up an isomorphism between  $\overline{\ob}_1$ and $\ob_2$ that is isotopic to the identity, and then derive the conclusion by reversing the orientation on the pages of  $\overline{\ob}_1$. 

\begin{lemma}  \label{lem: bin}   Under the identification of $ST\S$ with $V(\S)$, the binding  
of   $\overline{\ob}_1$ corresponds to the binding of  $\ob_2$. \end{lemma}

\begin{proof} The binding of  $\overline{\ob}_1$  (see Section~\ref{sec: ACampo} and Figures~\ref{F:ACampo1} and \ref{F:ACampo2}) is the set of unit vectors tangent to the divide $P$. Under the identification of $ST\S$ with $V(\S)$, each unit tangent vector represents a cooriented line, by taking its normal.  Now we show that the set of these cooriented lines normal to $P$ is the binding of  $\ob_2$ for~$V(\S)$. Recall that the characteristic surface  $C_{\widehat{X}}$ (see Section~\ref{sec: cano3}) consists of those cooriented lines containing the gradient-like vector  field $X$ for $f$, which can be assumed to be the gradient $\nabla f$ away from the critical points of $f$. Notice that we fixed $f$ from the beginning of the proof, but we are free to choose the gradient-like vector field $X$, as long as it is in a certain form near the critical points of $f$ as we explained in Section~\ref{sec: conv}.

The binding of $\ob_2$ is given by the intersection of  $f_\xi^{-1}(0)$  with $C_{\widehat{X}}$.   First of all, we observe that $f_\xi^{-1}(0)$ does not intersect the fibers above the Morse charts around index $2$ or index $0$ critical points. Away from the index $1$ critical points of $f$, the surface  $f_\xi^{-1}(0)$   consists of all the circle fibers above the divide $P=f^{-1}(0)$. Since $X=\nabla f$ away from the critical points, there are two points in each such circle fiber that belongs to $C_{\widehat{X}}$,  corresponding to the normal directions to $P$.  So, we conclude that, possibly except for a neighborhood of index $1$ critical points,
the binding of~$\ob_2$ consists of cooriented normal lines along $P$.  Nevertheless, one can check that this is  also the case on a Morse chart around each index $1$ critical point of $f$ either by looking at Figure~\ref{F:CriticalPoint} right, or by the following argument.

Using local coordinates $x_0, x_1$ for a Morse chart centered at an index $1$ critical point, and denoting the corresponding cotangent fiber coordinates  by $y_0=\cos \theta$ and $y_1=\sin \theta$, we have (cf. \cite{m}):
$$ f(x_0, x_1)=-x_0^2 +x_1^2$$
\begin{equation}\label{eq: vect} X=-x_0 \frac{\del}{\del x_0}+ x_1 \frac{\del}{\del x_1} \end{equation}
$$\xi=\ker (\cos \theta dx_0 -\sin \theta dx_1) $$
\begin{equation}\label{eq: mas}
f_\xi (x, y) = -x^2_0 +x_1^2 + \widetilde{\eta}  \chi (x) (y_0^2 - y_1^2) = -x^2_0 +x_1^2 + \widetilde{\eta} \chi (x) \cos 2\theta \end{equation}
$$\widehat{X}=-x_0 \frac{\del}{\del x_0}+ x_1 \frac{\del}{\del x_1} -\sin 2\theta  \frac{\del}{\del \theta}$$ where $\widetilde{\eta} >0$ is a sufficiently small real number.   Notice that $$\xi(\widehat{X})=0 \iff -x_0 \cos \theta = x_1 \sin \theta.$$ It follows that  the binding $f_\xi^{-1}(0) \cap  C_{\widehat{X}}$ consists of the following four disjoint line segments

$\theta=\pi/4$ and $-x_0=x_1$

$\theta=3\pi/4$ and $x_0=x_1$

$\theta=5\pi/4$ and $-x_0=x_1$

$\theta=7\pi/4$ and $x_0=x_1$

\noindent in the trivialized circle bundle over this Morse chart. Therefore, since $P=f^{-1}(0)$ is given by the lines $x_0 = \pm x_1$ in this chart, we conclude that the binding  consists of   cooriented normal lines along $P$, on this chart as well, which finishes the proof of the lemma.  \end{proof}

Next we show that the four distinct pages of $\overline{\ob}_1$ coincides with that of $\ob_2$, under the identification of $ST\S$ with $V(\S)$.  
The four pages we have in mind are the pages whose images are purely real or purely imaginary under the fibration map defining $\overline{\ob}_1$ given in Theorem~\ref{thm: ACampo},  namely the pages $S_0$, $S_{\pi/2}$, $S_\pi$ and  $S_{3\pi/2}$ of $\overline{\ob}_1$. 
\red{Of course it would be enough to show that these two open books have one page in common but it is not harder to prove that they share four pages.}

\begin{lemma} \label{lem: twopage1} Under the identification of $ST\S$ with $V(\S)$,  the union $S_0 \cup S_\pi$  corresponds to the characteristic surface $C_{\widehat{X}}$, which is the union of two pages in $\ob_2$. \end{lemma}

\begin{proof} The interior of the page $S_0$ (resp. $S_{\pi}$) of $\overline{\ob}_1$ consists of the unit vectors tangent to the level sets of $f$ in black (resp. white) regions of $\S \setminus P$ (see Figure~\ref{F:ACampo1}).  In particular,  the whole circle fiber above a critical point  of index $2$  (resp.  index~$0$) is included in $S_0$ (resp{.}~$S_\pi$).
We claim that the set of corresponding conormal lines of all the level sets in  black and white regions belongs to the characteristic surface~$C_{\widehat{X}} \subset V(\S)$, up to isotopy. This is clear outside of a neighborhood of the critical points, since the vector field $X$ is assumed to be equal to $\nabla f$, which is indeed normal to the level sets of $f$.  In addition, $X$ is normal to the level sets in a neighborhood of an index $1$ critical point by (\ref{eq: vect}), see Figure~\ref{F:CriticalPoint}.

On the other hand, over a critical point of index $2$ or index $0$, the whole fiber in $V(\S)$  is included in $C_{\widehat{X}}$.  In addition, by choosing the vector field $X$ to be sufficiently close to $\nabla f$ on a Morse chart around an  index $2$  (resp.  index $0$)  critical point, we can guarantee that the piece of $S_0$ (resp.  $S_\pi$) is isotopic to the piece of $C_{\widehat{X}}$ above that chart. In other words, we consider an isotopic fibration map defining the open book  $\overline{\ob}_1$, where the isotopy only takes place above the Morse charts around index $2$ or index $0$ critical points.  Therefore, this isotopy does not affect the discussion in the proof of Lemma~\ref{lem: bin}, since the binding in each open book does not intersect the fibers above the Morse charts around the critical points of index $2$ or index $0$. Hence, we finish the proof of the lemma, by adding in the bindings of both open books.  \end{proof}

\begin{lemma}  \label{lem: twopage2} Under the identification of $ST\S$ with $V(\S)$,  the union $S_{\pi/2} \cup S_{3\pi/2}$ corresponds to the level set  $f_\xi^{-1}(0)$, which is the union of two pages in $\ob_2$.  \end{lemma}

\begin{proof}  The interior of the page $S_{\pi/2}$ (resp. $S_{3\pi/2}$) of  $\overline{\ob}_1$ consists of unit vectors along the divide $P$ pointing into black (resp.  white) regions of $\S \setminus P$. Notice that $f_\xi^{-1}(0)$ does not intersect the fibers above the Morse charts around index $2$ or index~$0$ critical points.  Moreover,   away from the index $1$ critical points of $f$, the zero-level set of $f_\xi$ consists of all the circle fibers above the divide $P$ and thus all cooriented lines are in $f_\xi^{-1}(0)$,  where the binding of $\ob_2$ separates the two pages. To summarize,  away from the index $1$ critical points, the union $S_{\pi/2} \cup S_{3\pi/2}$ corresponds to the level set  $f_\xi^{-1}(0)$.

Near each index $1$ critical point, we understand both $f_\xi^{-1}(0)$ and the union $S_{\pi/2} \cup S_{3\pi/2}$  explicitly in local coordinates (again, see Figure~\ref{F:CriticalPoint} right). According to (\ref{eq: mas}), the surface $f_\xi^{-1}(0)$ is given locally by $$\{ (x_0, x_1, y_0, y_1) \; | \;  -x_0^2 + x_1^2 +\widetilde{\eta}  \chi(x) (y_0^2 -y_1^2) =0 \}$$ while according to Section~\ref{sec: ACampo}, the union $S_{\pi/2} \cup S_{3\pi/2}$  is given locally by $$\{ (x_0, x_1, u_0, u_1) \; | \;  -x_0^2 + x_1^2 -\dfrac{1}{2} \eta^2 \chi(x) (u_1^2 -u_0^2) =0\}.$$ So, we see that these surfaces coincide near each index one critical point,  by taking $\widetilde{\eta}=\eta^2/2$ and using identical cut-off functions $\chi : \S \to \R$. \end{proof}

We conclude the proof of Proposition~\ref{prop: 13} by Lemma~\ref{lem: twopage1} and Lemma~\ref{lem: twopage2}  that the four pages of  $\overline{\ob}_1$ coincides with the four pages of $\ob_3$ under the aforementioned diffeomorphism between  $ST\S$ and $V(\S)$. This is more than enough to conclude that $\ob_1$ is isotopic to $\ob_3$ by Proposition~\ref{prop: onepageisotopic}, after  reversing the orientation of the pages of $\overline{\ob}_1$ first. \end{proof}

\begin{lemma}\label{lem: 14} The \red{A'Campo-Ishikawa} open book $\ob_1$ and the \red{cogeodesic} open book $\ob_3$  for $ST^*\S$ are isotopic, provided that $P$ is convex with respect to the given metric on $\S$.
\end{lemma}

\begin{proof}
 The binding of both open books correspond to the link~$L(P)$ consisting of those vectors tangent to~$P$, hence they coincide. The geometric description in Section~\ref{sec: ACampo} of the page~$S_0$ of the \red{A'Campo-Ishikawa} open book $\overline{\ob}_1$  coincides with the page of the \red{geodesic} open book  $\overline{\ob}_3$ of  Birkhoff \& Fried given by Theorem~\ref{thm: bir}, based on a set of geodesics satisfying the assumptions of a divide (see the top left of Figure~\ref{F:ACampo1}).
\red{Indeed, one can take the foliation of~$\Sigma\setminus P$ described after Definition~\ref{def: convexedivide} for constructing~$\overline{\ob}_3$ and a Morse function~$f$ having the leaves of this foliation for~$\overline{\ob}_1$. 
Then} they both correspond to the set of vectors tangent to the level sets of~$f$ in the black regions of~$\S\setminus P$. 
By reversing the orientations, we get the desired isomorphism of $\ob_1$ and $\ob_3$ using Proposition~\ref{prop: onepageisotopic}.
\end{proof}

{In fact,  one can show that these two open books also have \emph{four} ``common" pages, namely the pages~$S_0, S_{\pi/2}, S_{\pi}$, and $S_{3\pi/2}$ from~$\ob_1$. }

\section{Lefschetz fibrations and open books of minimal genus}\label{sec: genus}

A particular admissible divide on a closed and oriented surface $\S$ of genus $g >1$ was given by Fried~\cite{f}, following Birkhoff~\cite{b}.
This divide consists of~$2g+2$ simple closed curves, with $2g+2$ double points, as we depicted at the top right of Figure~\ref{F:GenusOne}, for $g=3$.
Fried noticed that, when~$\S$ is hyperbolic, the corresponding Birkhoff cross section has genus one,  which can also be recovered from Proposition~\ref{prop: genus}.
The associated Lefschetz fibration $DT^*\S \to D^2$, whose fiber is of genus one ({\em minimal possible}) with $4g+4$ boundary components, was explicitly constructed by the second author in \cite{oz}.
Here we simply observe that this construction does not give the minimal number of boundary components of a fiber among all genus one Lefschetz fibrations, answering  Question (4) at the end of \cite{oz}  negatively.

\begin{proposition} \label{prop: boun}\footnote{This statement was independently proven by C. Bonatti (unpublished, private communication).}
 For any closed and oriented surface $\S$ of genus $g \ge 1$, there is an explicit genus one Lefschetz fibration $DT^*\S \to D^2$ whose fiber has $4g$ boundary components. As a consequence, there is an explicit genus one open book adapted to $(ST^*\S, \xi)$, whose binding has $4g$ components.
\end{proposition}

\begin{figure}[h!]
	\includegraphics[width=.82\textwidth]{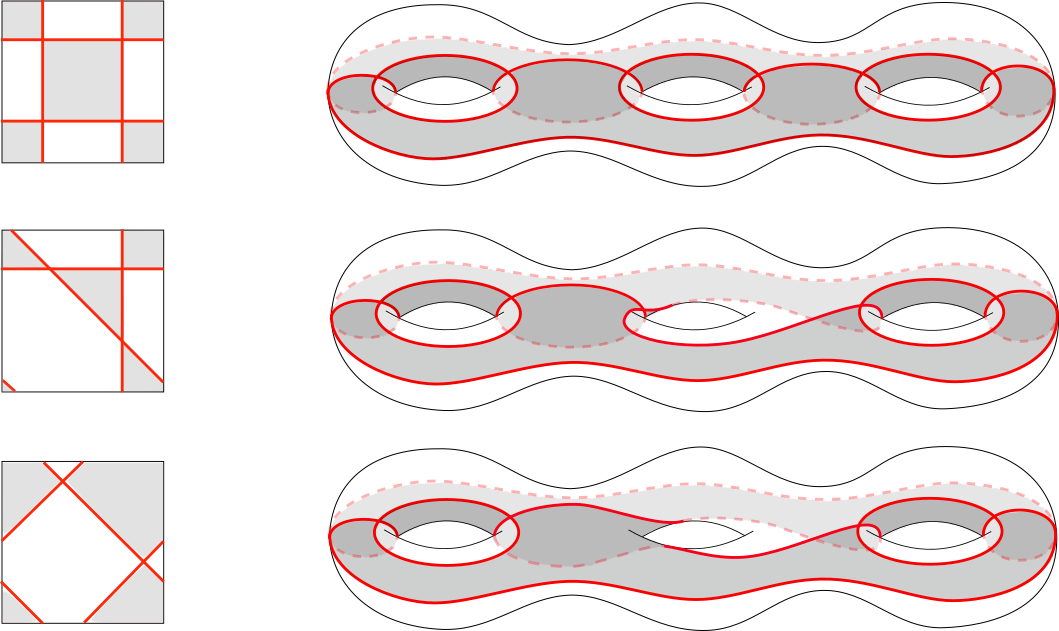}
	\caption{Three admissible divides that yield genus one Lefschetz fibrations on $DT^*\S$.  The top row corresponds to the Birkhoff-Fried divide which consists of~$2g+2$ curves that we depicted  for a torus on the left and on a genus $3$ surface on the right. The middle row corresponds to Brunella's improvement with~$2g+1$ curves. It is obtained by smoothing one crossing, thus connecting the two white regions. The bottom row corresponds to an additional improvement with~$2g$ curves only, obtained by smoothing another crossing and connecting the two black regions. }
	\label{F:GenusOne}
\end{figure}

\begin{proof} The set of $2g$ simple closed curves on $\S$ depicted at the bottom row of Figure~\ref{F:GenusOne} yields an admissible divide on~$\S$.
Since each of these curves intersect exactly two other curves, the regular fiber of the Lefschetz fibration $DT^*\S \to D^2$ obtained by Ishikawa's \red{method} is of genus one, with $4g$ boundary components. Moreover, the monodromy of this Lefschetz fibration can be explicitly computed as a product of Dehn twists. It follows that,  there is an explicit genus one  \red{A'Campo-Ishikawa} open book adapted to $(ST^*\S, \xi)$, whose binding has $4g$ components. Furthermore, by Theorem~\ref{thm: main}, this open book is isomorphic to the Giroux open book as well as the geodesic open book.    \end{proof}

An open book for a closed $3$-manifold induces a Heegaard splitting, where the Heegaard surface is the union of two pages along the binding. If the page of the open book is a surface of genus $h$ with $k$ boundary components, then the Heegaard surface is of genus $2h+k-1$. According to Boileau-Zieschang \cite{bz}, for any closed and oriented surface~$\S$ of genus $g$, the Heegaard genus of $ST^*\S$ is~$2g+1$. Therefore,  the binding of any {\em genus one} open book for $ST^*\S$ must have at least $2g$ connected components, which implies that the fiber of any {\em genus one} Lefschetz fibration $DT^*\S \to D^2$ must have at least  $2g$ boundary components. In other words, for an oriented surface $\S$ of genus at least two,  the binding number $\bn(ST^*\S, \xi)$ \cite{eo} satisfies the following inequality
$$ 2g \leq \bn(ST^*\S, \xi) \leq 4g.$$
This simple observation indicates that there is some room for improvement for the result stated  in Proposition~\ref{prop: boun}, using methods other than discussed in this paper.  For example, {a genus one open book  adapted to~$(ST^*T^2 = T^3, \xi)$ with {\em three} binding components, was constructed by Van Horn-Morris \cite{vhm} (see also \cite[Section~3]{dd} for similar examples obtained by the first author). }

{Next we observe that Proposition~\ref{prop: boun} cannot be improved utilizing the constructions described in this article. Notice that Theorem~\ref{thm: opt} below implies Theorem~\ref{thm: genusone} from the introduction.}

\begin{theorem}\label{thm: opt}
Up to homeomorphism, there are exactly three admissible divides (depicted in Figure~\ref{F:GenusOne}) on a genus~$g \geq 1$ surface~$\S$ that yield genus one Lefschetz fibrations~$DT^*\S\to D^2$.
\end{theorem}

\begin{proof}
Let $P$ be an admissible divide on~$\S$.
Recall that $c(P)$ denotes the number of curves that form~$P$ and $v(P)$ the number of double points in $P$.
By Proposition~\ref{prop: genus}, the genus of the associated Lefschetz fibration~$DT^*\S\to D^2$ is $1+v(P)-c(P)$.
Hence this genus is equal to~$1$ if and only if $c(P)=v(P)$.
Since~$P$ is admissible, its complement can be black-and-white colored, which means that every curve in~$P$ has an even number of double points (counting self-intersection points twice).
Since $P$ is connected, every component of $P$ has at least two double points.
By counting all double points twice (one for every curve that traverses the point), one obtains
$2v(P)\ge 2c(P)$, with equality if and only if every curve in~$P$ has exactly two double  points on it. One can check that this is indeed the case for all the divides depicted in Figure~\ref{F:GenusOne}.

Assume now that $P$ is an admissible divide so that every curve in~$P$ has exactly two double points on it.
Consider the {dual} graph~$\Gamma(P)$ whose vertices are the curves in~$P$ and whose edges are the double points. Note that ~$\Gamma(P)$ is a graph with loops and multiple edges.
Since $P$ is admissible, $\Gamma(P)$ is connected, for otherwise the complement~$\S\setminus P$ would have a non-trivial topology.
Our assumption on the number of double points implies that the vertices of~$\Gamma(P)$ all have degree~$2$.
Hence~$\Gamma(P)$ is a cycle, and therefore $P$ is {cyclic} chain of circles.

If one removes a small {open} disk in each region of~$\S\setminus P$, one obtains a surface with boundary, which deformation-retracts on a neighborhood~$\S_P$ of~$P$.
The surface with boundary~$\S_P$ is then a {cyclic} chain of circular ribbons, say~$r_1, \dots, r_k$, {where $k=c(P)$.}
One can recover~$\S_P$ by gluing one by one all the ribbons, starting with~$r_1$, then gluing~$r_2$, then $r_3$, and so on, inductively.
Finally one adds~$r_k$ and glues back~$r_k$ to~$r_1$.
Since~$\S_P$ is orientable, every ribbon is a standard annulus, and there is no freedom in the homeomorphism type of the gluing, except at the end when gluing~$r_k$ back to~$r_1$.
Indeed, as shown in Figure~\ref{F:Ribbon}, the open chain~$r_1\cup\dots\cup r_k$ has four boundary components, which are color-coded.
When gluing~$r_k$ back to~$r_1$, one has to chose which boundary component is glued with which one.

\begin{figure}[h!]
	\includegraphics[width=.92\textwidth]{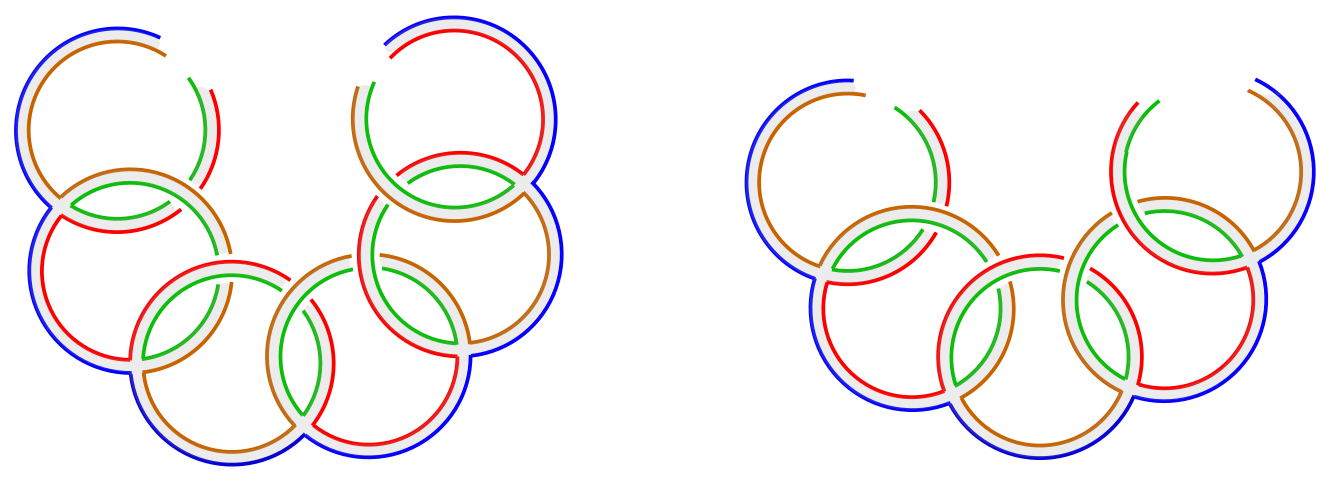}
	\caption{Proof of Theorem~\ref{thm: opt} is illustrated for even $k$ on the left and for odd $k$ on the right.  When reconstructing a neighborhood~$\S_P$ of~$P$ in~$\S$, one has a chain of ribbons~$r_1, \dots, r_k$.
	The surface~$\S_P$  with boundary is obtained by connecting the two open bands in both figures.
	When $k$ is even, there are two possibilities: either one glues each (colored) boundary with itself, or one glues blue with green, green with blue, red with brown, and brown with red.
	In the first case~$\S_P$ has four boundary components of length~$k$, while in the second it has two boundary components of length~$2k$.
	When $k$ is odd, there are also two possibilities: either one glues blue with blue, brown with brown, red with green, and green with red, or one glues blue with brown, brown with blue, red with red, and green with green. In either case~$\S_P$ has three boundary components, two of length~$k$ and one of length~$2k$. }
	\label{F:Ribbon}
\end{figure}

If $k$ is even, there are exactly two possible cases that lead to two different divides.
In the first case, one glues each boundary component with itself and therefore~$\S_P$ has four boundary components.
When filling these boundary components with disks, thus recovering~$\S$, one checks that~$\S$ is made of four~$
k$-gons glued along their edges.
It has genus~$\frac{k-2}2$, so that {$k = 2g+2$ and the divide~$P$ corresponds to the  Birkhoff-Fried divide~\cite{f} on~$\S$ (as depicted at the top row of Figure~\ref{F:GenusOne}).}
In the second case, one glues each boundary component with another one so that~$\S_P$ has two boundary components.
When recovering~$\S$, it is made of two $2k$-gons glued along their edges.
It has genus~$\frac k2$,  so that {$k = 2g$ and the divide~$P$ is the one depicted at the bottom row of Figure~\ref{F:GenusOne}.}

Finally when $k$ is odd, there is only one possibility up to homeomorphism:
two boundary components have to be glued to themselves, and the two others are glued one with another.
Thus~$\S$ is made of two $k$-gons and one $2k$-gon glued along their edges.
It has genus~$\frac{k-1}2$, so that {$k = 2g+1$ and the divide~$P$ corresponds to the Brunella divide depicted at the middle row of Figure~\ref{F:GenusOne}.}
\end{proof}

Now we turn our attention to the monodromies associated to the genus one geodesic open books obtained by restricting to the boundary of the three genus one Lefschetz fibrations in Theorem~\ref{thm: opt}. 
Since the pages have genus one, after contracting their boundary components into points, the monodromies correspond to isotopy classes of orientation-preserving homeomorphisms of the $2$-dimensional torus. 
\red{Such an isotopy class admits unique linear map, which is then described by an element in~$\SLZ$ (see for example~\cite[Thm 2.6.1]{kh}). 
Changing the basis of the torus conjugates this element, hence only the conjugacy class in~$\SLZ$ is relevant. 
It turns out that every non-elliptic conjugacy class in~$\SLZ$ contains a positive product of the matrices~$L=(\begin{smallmatrix} 1&0\\1&1\end{smallmatrix})$ and $R=(\begin{smallmatrix} 1&1\\0&1\end{smallmatrix})$, which is unique up to cyclic permutation. 
Summarizing, provided that they are not elliptic, the monodromies associated to the genus one open books are associated to word in~$L$ and~$R$ up to cyclic permutation.}

The monodromy of the Birkhoff-Fried geodesic open book was computed by Ghys~\cite{gh} and Hashiguchi~\cite{h}. Moreover, Brunella already computed  the monodromy of the geodesic open book induced by his divide~\cite{br}. 
Finally, the monodromy corresponding to the geodesic open book of Proposition~\ref{prop: boun} can be computed using similar techniques, or the method of Dehornoy-Liechti~\cite{dl}. Here is a summary of these computations:
\[
\begin{array}{c|c|c}
\textrm{geodesic~open~book} & \textrm{number~of~binding~components} & \textrm{monodromy}\\
\hline
\textrm{Birkhoff-Fried} & 4g+4& L^{g-1}R^2L^{g-1}R^2  \\
\hline
\textrm{Brunella} & 4g+2& L^{g-1}R^4L^{g-1}R^2 \\
\hline
\textrm{Prop~}\ref{prop: boun} & 4g& L^{g-1}R^4L^{g-1}R^4 \\
\hline
\end{array}
\]
\medskip

Cossarini and the first author \cite{cd} generalized the work of Birkhoff  \cite{b}, Fried \cite{f}  and Brunella \cite{br},  to construct and classify {\em negative}  Birkhoff cross sections of the geodesic flow associated to arbitrary Eulerian coorientations on a given finite set of geodesics on a hyperbolic surface $\S$. Moreover, Marty  \cite{mar} computed the monodromy of the corresponding open books for  $ST\S$ as a product of Dehn twists along curves explicitly described on a page.   These open books (when viewed for $ST^*\S$)  support the canonical contact structure $\xi$ as well. The following questions naturally arise as  a result of this discussion: Are these monodromy factorizations pairwise Hurwitz equivalent? Is the total space of the Lefschetz fibration filling each of these open books  diffeomorphic to $DT^*\S$? \\

\noindent {\bf {Acknowledgements}}: We are grateful to the anonymous referee whose suggestions helped us greatly improve the paper. 
We also thank Fran\c cois Laudenbach for exchanges concerning isotopic open books and Proposition~\ref{prop: onepageisotopic}. \\

\end{document}